\documentclass[10pt,a4paper]{article}
\usepackage[utf8]{inputenc}
\usepackage{amsmath}
\usepackage{amsfonts} 
\usepackage{amssymb}
\usepackage{amsmath,amsxtra,amssymb,latexsym,amscd,amsthm,amsfonts}
\newtheorem{theorem}{Theorem}
\newtheorem{cor}[theorem]{Corollary}

\newtheorem{exam} {Example}
\newtheorem{lem}[theorem]{Lemma}
\newtheorem{prop}{Proposition}

\newtheorem{rem}{Remark}

\DeclareMathOperator{\divv}{div}

\usepackage{color}
\begin{document}

\title{\bf Complete $\lambda$-submanifolds in Gauss spaces}
\author{\bf  Doan The Hieu\\
 Department of Mathematics\\
 College of Education, Hue University, Hue, Vietnam\\
 \\
 dthieu@hueuni.edu.vn}

\maketitle
\begin{abstract}

 In this paper, we study $\lambda$-submanifolds of arbitrary codimensions in  Gauss spaces. These submanifolds can be seen as natural generalizations of self-shrinker and $\lambda$-hypersurfaces. Using a divergence type theorem and some Simons' type identities, we prove some halfspace type theorems and  gap theorems for complete proper $\lambda$-submanifolds. These generalized  our as well as the others' results  for self-shrinker or $\lambda$-hypersurfaces to $\lambda$-submanifolds. 

 \end{abstract}

\vskip 1cm

\section{Introduction}
A manifold with density is a Riemannian manifold $M$  endowed with a positive function (density) $e^{-f}$  used to weigh both volume and perimeter. The weighted volume and perimeter elements are defined as $e^{-f} dV$ and $e^{-f} dA,$ where $dV$ and $dA$ are the Riemannian volume and perimeter elements. We refer the reader to \cite{mo1}, \cite{mo2}, \cite {RCBM}  for more details bout manifolds with density.

Manifolds with
density appeared in mathematics long ago and is a special case of “mm spaces”, studied earlier by Gromov \cite{gr}.
   Following Gromov (\cite[p. 213]{gr}),  the natural generalization of the mean curvature of a hypersurface $\Sigma$ in a manifold $M$ with density $e^{-f}$ is defined as
\begin{equation}      H_{f}=H+\langle \nabla f, {\bf n}\rangle,\end{equation}
where $H$ is the classical mean curvature and {\bf n} is the normal vector field of $\Sigma.$
$H_f$ is called weighted mean curvature, or curvature with density, or $f$-curvature of $\Sigma$ by some mathematicians nowadays. A hypersurface $\Sigma$ with $H_f= 0$ everywhere is called  weighted minimal or $f$-minimal. If  $H_f=\lambda \ (\text{a constant}),$  then $\Sigma$ is called a $\lambda$-hypersurface.

 A typical example of  manifolds with density is Gauss space $\mathbb G^m,$ that is $\Bbb R^m$ with Gaussian probability density $(2\pi)^{-\frac m2}e^{-\frac{|X|^2}2},$ where $X$ is the position vector. Gauss space has many applications to probability and statistics. In Gauss space, $f$-minimal submanifolds  are nothing but self-shrinkers, self-similar solutions to the mean curvature ﬂow (MCF) that play
an important role in the study  singularities of the MCF. The study of self-shrinkers and $\lambda$-hypersurfaces attracts the attention of many mathematicians.  For more information about
self-shrinkers as well as singularities of the MCF, we refer the readers to   \cite{comi3}, \cite{hu1}, \cite{hu2}, \ldots, and references therein.
For some results concerning $\lambda$-hypersurfaces, see \cite{anmi}, \cite{chwe}, \cite{gu}, \cite{waxuzh},  \cite{wepe}.

Some halfspace type theorems for self-shrinker of codimension 1 or for $\lambda$-hypersurfaces were proved in \cite{piri}, \cite{huparo}, \cite{caes}, \cite{impiri}, \cite{vizh} and for self-shrinkers of arbitrary codimension   in \cite{hidu}.

In 2011, Le-Seum \cite{lese} prove a gap theorem for self-shrinkers of codimension 1. Soon after, in 2012, Cao-Li \cite{cali} generalized the result to arbitrary codimension case (see also \cite{dixi1}). They proved
\begin{theorem} [\cite{cali}, Theorem 1.1]
If $\Sigma^n \rightarrow \mathbb R^{n+p} (p \ge 1)$ is an $n$-dimensional complete self-shrinker without
boundary, with polynomial volume growth, and satisfies
$$
|A|^2 \le 1,
$$
then $\Sigma$ is one of the following:
\begin{enumerate}
\item  a hyperplane in $\mathbb R^{n+1},$
\item a round sphere $S^n(\sqrt n)$ in $\mathbb R^{n+1},$
\item a cylinder $S^k(\sqrt k)\times \mathbb R^{n-k}, 1 \le k \le n - 1,$ in $\mathbb R^{n+1}.$

\end{enumerate}
Here $|A|^2$ is the squared norm  of the second fundamental form of $\Sigma.$
\end{theorem}

For $\lambda$-hypersurfaces, with density $e^{-\frac{|X|^2}4}$, Guang \cite{gu} prove the following classification theorem.
\begin{theorem} [\cite{gu}, Theorem 1.3]
If $\Sigma^n \subset \mathbb R^{n+1}$ is a smooth complete embedded $\lambda$-hypersurface with polynomial volume growth and satisfies
$$
|A| \le\frac{\sqrt{\lambda^2 + 2} - |\lambda|}2,
$$
then $\Sigma$  is one of the following:
\begin{enumerate}
\item a hyperplane in $\mathbb R^{n+1},$
\item a round sphere $S^n,$
\item  a cylinder $S^k \times \mathbb R^{n-k}$  for $1 \le k \le n - 1.$

\end{enumerate} 
\end{theorem}
And Wei-Peng \cite{wepe} proved
\begin{theorem} [\cite {wepe}, Theorem 1.2]
If $X: \Sigma \rightarrow \mathbb R^{n+1}$  is an $n$-dimensional complete $\lambda$-hypersurface with
polynomial area growth and satisfies $|A|^2$  bounded and
$$
H(H -\lambda)|A|^2\le  H^2, $$
then $\Sigma$ is one of the following:
\begin{enumerate}
\item a hyperplane $\mathbb R^n,$
\item a round sphere $S^n(r),$
\item  a cylinder $S^k(r) \times \mathbb R^{n-k}, \ 1\le  k \le  n - 1.$
\end{enumerate} 
Here $H$ is the mean curvature of $\Sigma.$
\end{theorem}

 In this paper, we generalize some  results (for self shrinkers or $\lambda$-hypersurfaces) including all theorems mentioned above and the others to $\lambda$-submanifolds.
We compute some Simon's type inequalities for $\lambda$-submanifolds and from there set up the conditions to infer the desired results. The computations in the higher codimensions case are, of course,  more complicated than the codimension 1 case. 

Different from the proofs of previous similar results (for self shrinker and $\lambda$-hypersurfaces), we use a divergence type theorem instead of integral estimates. That makes the proofs  simpler and shorter. 
 
\section{Preliminaries}
 \subsection{Definitions}
Let $\Sigma^n$ be a  submanifold in a  manifold $M^{n+p}$ endowed with a density $e^{-f}.$ Naturally, we can define the $f$-mean curvature vector of $\Sigma$ as (see \cite{impiri})
\begin{equation}\label{2}
{\bf H_f}={\bf H}+(\nabla f)^\perp
 \end{equation}
where ${\bf H}$ is the mean curvature vector of $\Sigma.$  If ${\bf H_f}\equiv 0,$ then $\Sigma$ is called a weighted or an $f$-minimal submanifold. If   $|{\bf H_f}|=\lambda,$ a constant, then $\Sigma$ is called a $\lambda$-submanifold. 

Weighted laplacian of a function $u:\Sigma\rightarrow \mathbb R,$ denoted by $\Delta_f(u)$ is defined as
$$
\Delta_f(u)=e^f\divv_\Sigma (e^{-f} \nabla_\Sigma(u).
$$
A direct computation shows that 
\begin{equation}\label{3}
\Delta_f(u)=\Delta_\Sigma u-\langle \nabla_\Sigma f, \nabla_\Sigma u\rangle.
\end{equation}

Since the density in Gauss spaces is
$(2\pi)^{-\frac m2}e^{-\frac{|X|^2}2},$ it is easy to see that
\begin{enumerate}
\item the equation for self-shrinkers is
$$
{\bf H}=-X^\perp;
$$
\item the $f$-mean curvature vector (\ref{2}) becomes
$$
{\bf H_f}={\bf H}+X^\perp;
$$
\item the equation for $\lambda$-submanifolds is
$$|{\bf H}_f|=|{\bf H}+X^\perp|=\lambda,$$
where  $\lambda\ge 0$ is a constant; 
\item and weighted laplacian (\ref{3}) becomes
$$
\Delta_f(u)=\Delta_\Sigma u-\langle X, \nabla_\Sigma u\rangle.
$$
\end{enumerate}
\begin{rem}
\begin{enumerate}
\item  $0$-submanifolds are self-shrinkers of codimension $p$. 
\item In Gauss spaces, the weighted Laplacian, $\Delta_f,$ is nothing but ${\cal L}$-operator introduced first by Colding and Minicozzi in \cite{comi3}. 
\item In the definition of a $\lambda$-submanifold, $\lambda\ge 0$ while in the definition of $\lambda$-hypersurface, $\lambda$ can be negative. But we can see that  a  $\lambda$-hypersurface is a $|\lambda|$-submanifold of codimension 1.
\end{enumerate}
\end{rem}

\subsection{Examples}
\begin{exam} [$n$-planes] \label{pla}
Let $\Sigma$ be an $n$-plane in $\mathbb G^{n+p}$ and $H$ be the orthogonal projection of the origin $O$ onto $\Sigma.$ The weighted mean curvature vector of $\Sigma$ is
\begin{align*}
{\bf H}_f&={\bf H}+X^\perp=X^\perp\\
&=\overrightarrow{OH}.
\end{align*}
Thus, $\Sigma$ is a $\lambda$-submanifold of codimension $p,$ where $\lambda=d(O, \Sigma).$ If $O\in\Sigma,$ then $\Sigma$ is a self-shrinker.
\end{exam}

\begin{exam} [Spheres with center $O$] \label{sph}
Let $S^{n}(r)$ be an $n$-sphere with center $O$ and radius $r$ in $\mathbb G^{n+p}.$  The mean curvature vector of $S^{n}(r)$ is
$${\bf H}=-\frac {n}{r^2}X.$$
Therefore,
\begin{align*}
{\bf H}_f&={\bf H}+X^\perp=-\frac {n}{r^2}X+X\\
&=\frac{r^2-n}{r^2}X.
\end{align*}
Thus, $S^n(r)$ is a $\lambda$- submanifold, where $\lambda=|\frac{r^2-n}{r}|$ and $r=\frac{\pm\lambda+\sqrt{\lambda^2+4n}}2.$ 
\end{exam}

\begin{exam} [Spheres with center $I\ne O$] \label{sph1}
Let $S^{n}(I, r)$ be an $n$-sphere center $I$ and radius $r$ in  $\mathbb G^{n+p}.$ Suppose that $\overrightarrow{OI}$ is perpendicular to the $(n+1)$-plane $P$ containing the sphere and   $d(O,I)=h.$ 
 
 The mean curvature vector of $S^{n}(I, r)$ is
$${\bf H}=-\frac {n}{r^2}\overrightarrow{IX}.$$
Therefore,
\begin{align*}
{\bf H}_f&={\bf H}+{X^\perp}={\bf H}+{X}\\
&=-\frac n{r}e_1+({re_1+he_2})\\
&=\left(-\frac n{r}+r \right)e_1+ h e_2,
\end{align*}
where $e_1=\frac{\overrightarrow{IX}}{|\overrightarrow{IX}|}.$  $e_2=\frac{\overrightarrow{OI}}{|\overrightarrow{OI}|}.$ 
Because 
\begin{align*}
{\bf H}_f^2=\left(-\frac n{r}+r\right)^2+h^2:=\lambda^2,\\
\end{align*}
$S^{n}(I, r)$ is a $\lambda$-surface, $\lambda=\sqrt{ \left(-\frac n{r}+r\right)^2+h^2}\ge h.$
\end{exam}
\begin{rem}
\begin{enumerate}
\item $h\le \lambda,$ means  the bigger $h$ is, the bigger $\lambda$ is.
\item It is not hard to show that if  $\overrightarrow{OI}$ is not perpendicular to the plane $P,$ $|{\bf H_f}|$ is not a constant, i.e. $S^n(I,r)$ is not a $\lambda$-submanifold.
\end{enumerate}

\end{rem}

\begin{exam}  [Cylinders] \label{cyl}

Let $S^{k}(I, r)$ be a $k$-sphere, $k<n,$ center $I$ and radius $r$ in  $\mathbb G^{n+p}$ and $P$ be the linear $(n-k)$-subspace perpendicular to the $(k+1)$-plane  containing $S^k.$ Consider the cylinder
${\cal C}=S^{k}(I,r)\times P$ and suppose that $I\in P$ (we can assume that $I\equiv O).$ 
In this case, 
$$|{\bf H}|=\frac {k}{r},\ \ \ |X^\perp|=r,$$
${\bf H}$ and $X^\perp$ are in opposite direction. Therefore
${\cal C}$ is a $\lambda$-submanifold, where $\lambda=|r-\frac kr|$ and $r= \frac{\pm\lambda+\sqrt{\lambda^2+4k}}2.$
\end{exam}
\begin{rem}
\begin{enumerate}
\item When $r=\sqrt k,$ $\Sigma$ is a self-shrinker. 
\item We can check that if  $P$ does not contain $I,$ $S^n(I,r)$ is not a $\lambda$-submanifold.
\end{enumerate}
\end{rem}

\begin{exam} [CMC submanifolds on spheres]  \label{min}
Let $\Sigma^n$ be a complete submanifold immersed in $S^{n+p-1}(r) \subset\mathbb G^{n+p}.$
We choose a local orthonormal frame $\{e_1, e_2, \ldots, e_{n+p}\}$ in $\mathbb G^{n+p}$ such that $\{e_1, e_2, \ldots, e_n\}$ are tangent to $\Sigma,$ $\{e_{n+1}, e_{n+2},$ $ \ldots, e_{n+p-1}\}$ are in the normal bundle of $\Sigma$  in $S^{n+p-1}(r),$ and $e_{n+p}$ is outward normal to $S^{n+p-1}(r).$
Levi-Civita connections on $S^{n+p-1}(r)$  (on $\mathbb G^{n+p}$) is denoted by  ${\nabla}$  (by $\overline{\nabla})$ and the mean curvature vector of $\Sigma$ in $S^{n+p-1}(r)$  (in $\mathbb G^{n+p}$) is denoted by  ${\bf H}$   (by $\overline{\bf H}).$ We have
\begin{align*}
\overline{\bf H}&=\sum^n_{i=1}\sum^{n+p-1}_{j=n+1}\langle\overline{\nabla}_{e_i}e_i , e_j\rangle e_j+\sum^n_{i=1}\langle\overline{\nabla}_{e_i}e_i , e_{n+p}\rangle e_{n+p}\\
&=\sum^n_{i=1}\sum^{n+p-1}_{j=n+1}\langle{\nabla}_{e_i}e_i , e_j\rangle e_j +\sum^n_{i=1}\langle\overline{\nabla}_{e_i}e_i , e_{n+p}\rangle e_{n+p}\\
&={\bf H}+\langle\sum^n_{i=1}\overline{\nabla}_{e_i}e_i , e_{n+p}\rangle e_{n+p}.
\end{align*} 
Since $\overline{\nabla}_vX = v$ for every vector $v,$  and $e_{n+p}=\frac  1rX=\frac  1rX^\perp,$
$$\langle\overline{\nabla}_{e_i}e_i, e_{n+p}\rangle =-\langle e_i ,\overline{\nabla}_{e_i}\frac 1{r}X\rangle=-\langle e_i,\frac 1{r}e_i\rangle=-\frac 1{r}, \ \ i=1,2,\ldots, n.$$
 Therefore,
$$
\overline{\bf H}={\bf H}-\frac {n}{r}e_{n+p}={\bf H}-\frac {n}{r}\frac{X}{r}={\bf H}-\frac {n}{r^2}X,
$$
i.e.
$$
\overline{\bf H_f}=\overline{\bf H}+X^\perp={\bf H}+\left[\frac {r^2- n}{r}\right]e_{n+p}.
$$
Thus,
\begin{enumerate}
\item $\Sigma$ is a $\lambda$-submanifold if and only if it is a CMC submanifold of $S^{n+p-1}(r)$.
\item \label{11} If $r=\frac{\pm\lambda+\sqrt{\lambda^2+4n}}2,$ $\lambda$-submanifolds are minimal submanifolds of $S^{n+p-1}(r)$ and vise versa.
\end{enumerate}
\end{exam}
\begin{exam} [Product of $\lambda$-submanifolds] \label{pro}
Let $\Sigma_1^{n_1}$ and $\Sigma_2^{n_2}$ be submanifolds  in $\mathbb G^{N_1}$ and $\mathbb G^{N_2},$ respectively. Consider the product  $\Sigma_1\times\Sigma_2\subset \mathbb G^{N_1+N_2}.$ We can see that
$${\bf H}={\bf H}_1+{\bf H}_2,$$
and
$${\bf H}_f={\bf H}_{1,f}+{\bf H}_{2,f};$$
where ${\bf H},  {\bf H}_1, {\bf H}_2$ are mean curvature vectors and
${\bf H}_f,  {\bf H}_{1,f}, {\bf H}_{2,f}$ are weighted mean curvature vectors of 
$\Sigma_1\times\Sigma_2, \Sigma_1$ and $\Sigma_2,$ respectively.
It follows that if $\Sigma_1$ is a $\lambda_1$-submanifold and $\Sigma_2$ is a $\lambda_2$-submanifold, then  $\Sigma_1\times\Sigma_2\subset \mathbb R^{N_1+N_2}$ is a $(\lambda_1+\lambda_2)$-submanifold.

Typical examples of products of $\lambda$-submanifolds are cylinders (see Example \ref{cyl}), Clifford tori in higher dimension or more generally, products of spheres. 
\end{exam}

\subsection{Properness vs. polynomial volume growth}

The volume growth is an important property for a complete non-compact manifold. 
An $n$-dimensional submanifold in $\mathbb R^{n+p}$ has Euclidean  volume growth if there exist a constant $C$ such that for all
$r \ge 1,$
$$
{\rm Vol}(B(r)) \cap \Sigma) = \int_{B(r)\cap\Sigma} d\mu \le Cr^n.
$$
It is said to have polynomial volume growth if there exist constants $C$ and $d$ such that for all
$r \ge 1,$
$$
{\rm Vol}(B(r)) \cap \Sigma) = \int_{B(r)\cap\Sigma} d\mu \le  Cr^d,
$$
 where $B(r)\subset\mathbb R^{n+p}$ is the  ball with radius $r$ and centered at the origin.
For the case of codimension 1, Cheng-Zhou (\cite{chzh}, Theorem 1.3) proved that 
 a complete non-compact properly immersed self-shrinker $\Sigma^n$
in $\mathbb R^{n+1}$ is proper if and only if it has Euclidean volume growth (and therefore polynomial)  at most. The same result holds for the case of $\lambda$-hypersurfaces (\cite{chwe}, Theorems 5.1, 5.2) and for self-shrinkers of arbitrary codimension (\cite{dixi}, Theorem 1.1). 

Halldorsson (\cite{ha}, Theorem 5.1) has proved that
there exist complete self-shrinker curves  $\Gamma$ in $\mathbb R^2,$ which is contained in an annulus around
the origin and whose image is dense in the annulus.  Since such a complete self-shrinker curve  $\Gamma$ is not proper, $\Gamma\times \mathbb R^{n-1}$ is a complete self-shrinker in $\mathbb R^{n+1},$ which does not have
polynomial volume growth.

We observe that Cheng-Wei's proof of Theorem 5.1 in \cite{chwe} also applies for  $\lambda$-submanifolds with a little change.  For the sake of completeness, we include the proof here.
\begin{prop}

Let $\Sigma^n\subset  \mathbb G^{n+p}$ be a complete and non-compact properly immersed $\lambda$-submanifolds. Then, there is a positive constant $C$ such that for $r \ge 1,$
$$
{\rm Area}(B_r(0)\cap X(M)) =\int_{B_r(0)\cap X(M)}dA \le  Cr^{n+\frac{\lambda^2}2 -2\beta-{\frac{\inf H ^2}2}},
$$
where $\beta = \frac 14 \inf|{\bf H_f}-{\bf H}|^2.$
\end{prop}

\begin{proof}
  Let $h = \frac{|X|^2}4-\beta$ and  $k=\frac n2+\frac{\lambda^2}4-\beta-\frac{\inf{H}^2}4.$ We will check that $h$ satisfies the conditions ($h- |\nabla_\Sigma h|^2\ge 0$ and $\Delta_hh+h\le k$) of Theorem 2.1 in \cite{chzh}. So by this theorem, the proof is done. Indeed,
\begin{enumerate}
\item\begin{align*}
h- |\nabla_\Sigma h|^2 &= \frac {|X|^2}4 - \beta-\frac{|X^T|^2}4\nonumber\\
&=\frac {|X^\perp|^2}4-\beta\\
&= \frac 14|{\bf H_f} - {\bf H}|^2-\beta\ge 0.\nonumber
\end{align*}
\item Since
$$
\Delta_\Sigma h=\Delta_\Sigma \frac{|X|^2}4=\frac n2+\frac 12\langle{\bf H}, X^\perp\rangle,
$$
we have
\begin{align*} \Delta_hh+h&=\Delta_\Sigma h-|\nabla_\Sigma h|^2+h\\
&=\frac n2+\frac 12\langle{\bf H}, X^\perp\rangle+\frac 14|{\bf H_f} - {\bf H}|^2-\beta\\
&=\frac n2+\frac 12\langle{\bf H}, {\bf H_f} -{\bf H}\rangle+ \frac 14|{\bf H_f} - {\bf H}|^2-\beta\\
&=\frac n2+\frac {\lambda^2}4-\beta-\frac {H^2}4 \\
&\le \frac n2+\frac{\lambda^2}4-\beta-\frac{\inf{H}^2}4=k.\\
\end{align*}
\end{enumerate}
\end{proof}
\subsection{A divergence type theorem} 
Let $\Sigma$  be an $n$-dimensional complete (without boundary)  $\lambda$-submanifolds properly, i.e. has polynomial volume growth,  immersed  in $\mathbb R^{n+p} .$ 

The condition of polynomial volume growth is essential for using an integral formula that is similar to the generalized divergence theorem for compact manifolds. We have the following theorem.
 \begin{theorem}\label{basis}
Let $u$ be a smooth function on $\Sigma.$ Assume that there exist positive constants $C$ and $d$ such that
 $ |\Delta_\Sigma u(X)|\le C|X|^{d}.$ Then
$$
\int_\Sigma e^{\frac{|X|^2}2}\Delta_f u(X)dV=0.
$$
\end{theorem}

\begin{proof}
Suppose that $\Sigma$ is inside a ball, since it is proper it must be compact and the theorem holds true by divergence theorem. 
Now suppose that $\Sigma$ is not inside any ball, i.e. 
$\partial(B_R\cap\Sigma)\ne\emptyset$ when $R$ is large enough. We have
\begin{align*}
\int_{\partial(B_R\cap\Sigma)} e^{\frac{|X|^2}2}\Delta_f u(X)dV&=\int_{B_R\cap\Sigma} \divv_\Sigma (e^{-\frac{|X|^2}2}\nabla_\Sigma u(X))dV\\
&=e^{-\frac{R^2}2}\int_{\partial(B_R\cap\Sigma)}\left\langle \nabla_\Sigma u(X),\nu\right\rangle dA.
\end{align*}

 Because
\begin{align*}
\lim_{R\rightarrow\infty}e^{-\frac{R^2}2}\left |\int_{\partial(B_R\cap\Sigma)}\left\langle \nabla_\Sigma u(X),\nu\right\rangle dA\right |&=\lim_{R\rightarrow\infty}e^{-\frac{R^2}2}\left|\int_{B_R\cap\Sigma}\Delta_\Sigma u(X) dV\right|\\
&\le \lim_{R\rightarrow\infty}e^{-\frac{R^2}2}CR^{d}\int_{B_R\cap\Sigma} dV\\
&\le \lim_{R\rightarrow\infty}e^{-\frac{R^2}2}C_1CR^{d+n}=0.
\end{align*}
the theorem is proved
\end{proof}

\section{Halfspace type Theorems}
In this section, $\Sigma$ is always assumed to be an $n$-dimensional complete (without boundary)  $\lambda$-submanifold  properly immersed  in $\mathbb G^{n+p}, p\ge 1.$ 

Halfspace type theorems in this section can be seen as an application of the Theorem \ref{basis} and Lemma \ref{form} bellow.

Let $e_1,e_2, \ldots, e_{n+p}$ be the coordinate vector fields for $\mathbb G^{n+p}$  and $X=(x_1,x_2,$ $\ldots, x_{n+p})$ be the position vector field.  By a straightforward computation (see \cite{hidu} for the case $\lambda=0$), we have the following lemma.
\begin{lem}\label{form}
\begin{enumerate}
\item  
\begin{align}
\Delta_\Sigma x_i&=-x_i|e_i^\perp|^2+\lambda_i;\label{4}\\
  \Delta_f x_i&=-x_i+\lambda_i.\label{5}
\end{align}
\item 
\begin{align}\label{6}
\Delta_\Sigma\frac{|X|^2}2&=n+\sum\lambda_ix_i-|X^\perp|^2;\\
 \Delta_f\frac{|X|^2}2&=n+\sum\lambda_ix_i-|X|^2.\label{7}
\end{align}
\item 
\begin{align}\label{8}
\Delta_\Sigma\frac {x_i^2}2
&=|e_i^T|^2+\lambda_ix_i-\frac12x_i^2|e_i^\perp|^2;\\ 
\Delta_f\frac {x_i^2}2&=|e_i^T|^2+\lambda_ix_i-\frac12x_i^2.\label{9}
\end{align}
\end{enumerate}
\end{lem}


\subsection{Halfspace type Theorem w.r.t. hyperplanes}

\begin{theorem}\label{hype}
Let  $P$ be a hyperplane in $\mathbb G^{n+p},$  such that $d(O,P)=\lambda.$ If  $\Sigma$ lies in the side of $P$ not containing the origin, then $\Sigma\subset P.$ \end{theorem}
\begin{proof}
Without loss of generality, we can suppose that $P$ is the hyperplane $x_{n+p}=\lambda$ and $\Sigma$ is in the closed half space $\{(x_1, x_2,\ldots, x_{n+p}):x_{n+p}\ge\lambda\}.$
By (\ref{4}) and (\ref{5})
\begin{align*}
\Delta_\Sigma x_{n+p}&=-x_{n+p}|e_{n+p}^\perp|^2+\lambda_{n+p},\\
  \Delta_f x_{n+p}&=-x_{n+p}+\lambda_{n+p},
\end{align*}
where $\lambda_{n+p}=\langle{\bf H_f},e_{n+p}\rangle.$ We can check that  $x_{n+p}$ satisfies the condition in Theorem \ref{basis} and because  $x_{n+p}\ge\lambda$ on $\Sigma,$  we get
\begin{align*}
0\le\int_\Sigma e^{-\frac{X^2}2}(x_{n+p}-\lambda)dV\le\int_\Sigma e^{-\frac{X^2}2}(x_{n+p}-\lambda _{n+p})dV=0.
\end{align*}
 It follows that $x_{n+p}=\lambda,$  i.e. $\Sigma\subset P.$
\end{proof}

\begin{rem} 
\begin{enumerate}
\item  If $p=1,$ then $\Sigma= P$ (see Theorem 1.4 in \cite{caes}).
\item  The case $\lambda=0$ was proved in \cite{hidu} (Theorem 7).
\end{enumerate}
 \end{rem}
 With the same arguments as in the proof of Corollary 9-10 in \cite{hidu} we have
\begin{cor} \label{co10}
If there exist $p$ orthonormal vectors $v_1, v_2,\ldots, v_{p}$ such that for $i=1,2,\ldots, p, |\langle X, v_i\rangle|\ge \lambda,$  then $\Sigma$ is an $n$-plane P, with $d(O,P)=\lambda.$
\end{cor}

\begin{cor} [A Bernstein type theorem] 
Let  $F=(f_1, f_2, \ldots, f_p) :\mathbb G^n\rightarrow \mathbb G^{p}$  be a smooth function and $\Sigma=\{({\bf x}, F({\bf x}))\in \mathbb G^{n+p}: {\bf x}\in \mathbb G^n\}$ be its graph that is a $\lambda$-submanifold. If there exist at least $(p-1)$ functions $f_i$ such that $|f_i|\ge \lambda,$ then $\Sigma$ is an $n$-plane.
\end{cor}
\subsection{$\lambda$-submanifolds inside or outside a ball}
Denote by $B^{n+p}(r)=\{x\in\mathbb G^{n+p}: x^2<r\}$  the standard ball with radius $r$ and $E^{n+p}(r)=\{x\in\mathbb G^{n+p}: x^2\ge r\}$ the complement of $B^{n+p}(r).$

\begin{theorem}\label{theoball}
If $\Sigma^n\subset E^{n+p}(r_2),$  where $r_2=\frac{\lambda+\sqrt{\lambda^2+4n}}2,$ then $\Sigma$ is compact and $\Sigma\subset S^{n+p-1}(r_2),$ i.e. $\Sigma$ is a minimal submanifold of $S^{n+p-1}(r_2).$  Moreover, if $p=1,$ then $\Sigma= S^{n}(r_2).$ 
\end{theorem}
\begin{proof}
From (\ref{6}) we can see that we can apply Theorem \ref{basis} for function $\frac{|X|^2}2,$  and by (\ref{7})
\begin{equation}\label{14}
\int_\Sigma e^{-\frac{X^2}2}(n+\lambda|X|-|X|^2)dV\ge\int_\Sigma e^{-\frac{X^2}2}(n+\sum\lambda_ix_i-|X|^2)dV=0.
\end{equation}

If $\Sigma\subset {E^{n+p}(r_2)},$  then $n+\lambda|X|-|X|^2\le 0$ . From (\ref{14}), it follows that  $n+\lambda|X|-|X|^2= 0,$ i.e. $\Sigma\subset S^{n+p-1}(r_2)$ and it is a minimal submanifold of  $S^{n+p-1}(r_2)$ (see Example  \ref{min}). Since $\Sigma$ is proper, it must be compact.
The case $p=1$ is obvious.
\end{proof}
Because 
\begin{equation}
|X|\ge |X^\perp|=|{\bf H}-{\bf H_f}|.
\end{equation}
we have
\begin{cor} If
\begin{equation} \label{332}
|{\bf H}-{\bf H_f}| \ge \frac {\lambda+\sqrt{\lambda^2+4n}}2, 
\end{equation}
then $|{\bf H}-{\bf H_f}| =\frac {\lambda+\sqrt{\lambda^2+4n}}2,$ and $\Sigma$ is a minimal submanifold in the sphere $S^{n+p-1}(\frac {\lambda+\sqrt{\lambda^2+4n}}2).$
\end{cor}
\begin{rem}
When $\lambda=0,$ (\ref{332}) becomes $H\ge \sqrt{n},$ the corollary was proved in \cite{cali} (Proposition 5.1).
\end{rem}

Since 
$$
n-\lambda|X|-|X|^2\le n+\sum\lambda_ix_i-|X|^2,
$$
by a similar proof, we have 
\begin{theorem}\label{theoball2}
If $\Sigma\subset \overline {B^{n+p }(\sqrt{r_1})},$ where $r_1=\frac{-\lambda+\sqrt{\lambda^2+4n}}2,$ then $\Sigma$ is compact, $\Sigma\subset S^{n+p-1}(r_1),$ and it is a minimal submanifold of $S^{n+p-1}(r_1).$  Moreover, if $p=1,$ then $\Sigma= S^{n}(r_1).$ 
\end{theorem}
\begin{rem}
\begin{enumerate}
\item The case $\lambda=0,$ i.e. $r_1=r_2$ was proved in \cite{hidu} (Theorem 11).
\item The following example shows that
the case $\lambda\ne 0$ is quite different with the case $\lambda=0.$ There exists a $\lambda$-submanifold is outside $S^n(r_1)$ and inside $S^n(r_2).$

In $\mathbb G^6$ consider $\Sigma=S^4(I, 2)\subset \{(x_1, x_2, \ldots, x_6: \ x_6=3\},$ where $I(0,0,\ldots, 0, 3).$ In this case $n=4$ and $\Sigma$ is a $\lambda$-submanifold, $\lambda=3,$  lying on the sphere $S^5(\sqrt{13}).$
It is easy to check that $\Sigma$ is outside the sphere $S^5(r_1)$ and inside the sphere $S^5(r_2),$ where
$$r_1=\frac{-\lambda+\sqrt{\lambda^2+4n}}2=1,$$
$$r_2=\frac{\lambda+\sqrt{\lambda^2+4n}}2=4.$$

\end{enumerate}
\end{rem}

 \subsection{Halfspace type results w. r. t.  cylinders}
  
  \begin{theorem} 
Let $k \in \{p, p+1,$ $ . . . , n+p-2\}, q=n+p-k-1$ and $\displaystyle r=\frac{-\lambda+\sqrt{\lambda^2+4(n-q)}}2.$ If  $\Sigma\subset \overline{B^{k+1}(r)}\times \mathbb R^q,$ then $\Sigma\subset S^{k}(r)\times \mathbb R^q.$ 
\end{theorem} 
\begin{proof}
From (\ref{8}), we can see that the function  $\frac{x_i^2}2$ satisfies the condition in Theorem \ref{basis}. Therefore,
\begin{equation}\label{141}
\int_\Sigma e^{-\frac{|X|^2}2}(x_i^2-\lambda_ix_i)dV=2\int_\Sigma e^{-\frac{|X|^2}2}|e_i^T|^2dV.
 \end{equation}

Suppose that  $\{e_1, e_2, \ldots, e_{k+1}\}\subset\mathbb R^{k+1}$ and $\{e_{k+2}, e_{k+3}, \ldots, e_{n+p}\}\subset\mathbb R^{q}.$
Write $X=(u, v),$ where $u\in \mathbb R^{k+1}, v\in\mathbb R^q.$

Combining  (\ref{14}) and (\ref{141}), we get
\begin{align*}
0&\le \int_\Sigma e^{-\frac{|X|^2}2}\left[q-\sum_{i=k+2}^{n+p}  |e_i^T|^2\right]dV\\
&=\int_\Sigma e^{-\frac{|X|^2}2}\left[q+|X|^2-\sum_{i=1}^{n+p}\lambda_ix_i-n+\sum_{i=k+2}^{n+p}\lambda_ix_i-\sum_{i=k+2}^{n+p} x_i^2\right]dV\\
&=\int_\Sigma e^{-\frac{|X|^2}2}\left[\sum_{i=1}^{k+1} x_i^2-\sum_{i=1}^{k+1}\lambda_ix_i-n+q\right]dV\\
&\le\int_\Sigma e^{-\frac{|X|^2}2}\left[|u|^2+\lambda|u|-(n-q)\right]dV.\\
\end{align*}
The assumption that $\Sigma\subset \overline{B^{k+1}(r)}\times \mathbb R^q,$  means
$|u|\le r,$ i,e. $|u|^2+\lambda|u|-(n+q)\le 0.$ It implies that
$|u|=r,$
 i.e. $\Sigma\subset S^{k}(r)\times \mathbb R^q .$ 
 
\end{proof}
\begin{rem}
 \begin{enumerate}
 \item We see in the above proof that  $e_i^\perp=0,$ i.e. $e_i=e_i^T, i=k+2,\ldots, n+p.$ Therefore, $\Sigma=\Gamma\times \mathbb R^q,$ where $\Gamma\subset S^{k}(r)$ is an $(n-q)$-dimensional $\lambda$-submanifold, i.e. an $(n-q)$-dimensional  CMC submanifold of $S^{k}(r).$ 

\item If $\lambda=0,$ this is Theorem 14 in \cite{hidu}.
\item If $p=1,$ we obtain Theorem 1.5 in \cite{caes}.
  \end{enumerate}
 \end{rem}
If we use the inequality
\begin{align*}
0&\ge-\int_\Sigma e^{-\frac{|X|^2}2}\sum_{i=k+2}^{n+p}  |e_i^T|^2dV\\
&=\int_\Sigma e^{-\frac{|X|^2}2}\left[|X|^2-n-\sum_{i=1}^{n+p}\lambda_ix_i+\sum_{i=k+2}^{n+p}\lambda_ix_i-\sum_{i=k+2}^{n+p} x_i^2\right]dV\\
&\ge\int_\Sigma e^{-\frac{|X|^2}2  }\left[|u|^2-\lambda|u|-n\right]dV,\\
\end{align*}
then by the same arguments as in the above proof, we have
 
\begin{theorem}
Let $k \in \{1, 2,. . . , n+p-2\}.$  If $\Sigma\subset\overline{E^{k+1}(r)}\times \mathbb R^{n+p-k-1},$ where $\displaystyle r=\frac{\lambda+\sqrt{\lambda^2+4n}}2,$ then $\Sigma\subset S^{k}(r)\times \mathbb R^{n+p-k-1}.$
\end{theorem} 

\section{Rigidity  and gap results } 
In this section $\Sigma^n \subset \mathbb G^{n+p}, p\ge 1,$ is an $n$-dimensional complete proper $\lambda$-submanifold without
boundary and  $\lambda>0.$

We choose a local orthonormal frame field $\{e_A\}_{A=1}^{n+p}$ in $\mathbb G^{n+p}$ with
dual coframe field $\{\omega_A\}_{A=1}^{n+p},$ such that when restricted to $\Sigma, e_1, \ldots , e_n$ are tangent to $\Sigma$ and $e_{n+p}=\frac{\bf H_f}{|{\bf H_f}|}.$  The indices will be used in the paper as follow:

$$1 \le A, B, C, D \le n+p, \ \ 1 \le i, j, k, l \le n,\ \  n + 1 \le \alpha, \beta, \gamma \le n+p.$$ Denoted by
\begin{enumerate}
\item $A =\sum_{\alpha,i, j}h^\alpha_{ij}\omega_i\otimes\omega_j\otimes e_\alpha,$ the second fundamental form;
\item $A^\alpha =\sum_{i,j}h^\alpha_{ij}\omega_i\otimes\omega_j\otimes e_\alpha,$ the second fundamental form corresponding to $e_\alpha;$
\item $H^\alpha=\sum_i h_{ii}^\alpha,$ the mean curvature corresponding to $e_\alpha;$
\item ${\bf H} = \sum_\alpha H^\alpha e_\alpha = \sum_a\left(\sum_i h_{ii}^\alpha\right) e_\alpha,$ the
mean curvature vector field of $\Sigma; $
\item $|A^\alpha|^2=\sum_{i,j}(h_{ij}^\alpha)^2,$  the squared norm of $A^\alpha;$
\item $|A|^2=\sum_{\alpha,i,j}(h_{ij}^\alpha)^2,$  the squared norm of $A;$
\item  $h^\alpha_{ijk}=\nabla_kh^\alpha_{ij}, \ \ h^\alpha_{ijkl}=\nabla_l\nabla_k h^\alpha_{ij},$ where $\nabla$ is the Levi-Civita connection on $\Sigma.$
\end{enumerate}
We have known that (see, e.g., \cite{lihai}), \cite{cali}, \cite{chpe}, \cite{si})

$$h_{ij}^\alpha = h_{ji}^\alpha.$$
 
 $$h_{ijk}= h_{ikj}.$$
The latter equality is the Codazzi equation.

First, we establish some Simon's type identities for $\lambda$-submanifolds that we need for proving  gap theorems.

\begin{lem}\label{lem15} \begin{enumerate}
 \item \begin{equation}
\Delta H^2=2|\nabla H|^2+2 H^2-2\sum_{\alpha,\beta,i,k}H^\alpha (H^\beta-\lambda_\beta) h^\alpha_{ik} h^\beta_{ik}+2\sum_{\alpha,k}H^\alpha H^\alpha_{,k}\langle X, e_k\rangle.
  \end{equation}
\item \begin{equation}\label{lem1.2}
\Delta_fH^2 =2|\nabla H|^2+2H^2-2\sum_{\alpha,\beta}H^\alpha (H^\beta-\lambda_\beta)\langle A^\alpha, A^\beta\rangle.
\end{equation}
 \item \begin{align}
\Delta|A|^2 &= 2 |\nabla A|^2 + 	\langle {\bf H}_{,ij}, A_{ij}\rangle + \langle A_{ik}, {\bf H}\rangle\langle A_{il}, A_{kl}\rangle-\sum_{\alpha\ne\beta}|[A^{e_\alpha}, A^{e_\beta}]|^2-\sum_{\alpha,\beta}S^2_{\alpha\beta}.
\end{align}

\item \begin{equation}\label{44}
\Delta_f|A|^2 = 2|\nabla A|^2 +2|A|^2 +2\langle   {\bf H_f}, A_{ik}\rangle A_{jk}, A_{ij}\rangle - 2 \sum_{\alpha\ne\beta}|[A^{e_\alpha}, A^{e_\beta}]|2 - 2 \sum_{\alpha,\beta}S^2_{\alpha\beta}.
\end{equation}

\end{enumerate}
\end{lem}
\begin{proof}
\begin{enumerate}
\item
Since ${\bf H_f}={\bf H}+X^\perp,$ 
\begin{equation}\label{19}
H^\alpha=-\langle X, e_\alpha\rangle+\lambda_\alpha,
\end{equation}
where $\lambda_\alpha=0, \ \alpha\ne n+p$ and $\lambda_{n+p}=\lambda$ (because $e_{n+p}=\frac{\bf H_f}{|{\bf H_f}|}).$

Taking covariant derivative of (\ref{19}) with respect to $e_i$ 
\begin{equation}\label{20}H^\alpha_{,i}= \sum_k h^\alpha_{ik}\langle X, e_k\rangle,\end{equation}
and taking covariant derivative of (\ref{20}) with respect to $e_j$ 
\begin{align}\label{47}
H^\alpha_{,ij} &=\sum_k h^\alpha_{ikj}\langle X, e_k\rangle+ h^\alpha_{ij}+\sum_{\beta, k} h^\alpha_{ik} h^\beta_{kj}\langle X, e_\beta\rangle\nonumber\\ 
&=\sum_k h^\alpha_{ijk}\langle X, e_k\rangle+ h^\alpha_{ij}+\sum_{\beta, k} (-H^\beta+\lambda_\beta ) h^\alpha_{ik} h^\beta_{kj},
\end{align}
we have 
\begin{align*}
\Delta H^2&=  2|\nabla H|^2 +2\sum_{\alpha,i}H^\alpha H^\alpha_{,ii}\\ 
&=2|\nabla H|^2 +2\sum_{\alpha}H^\alpha\left(\sum_{k}H^\alpha_{,k}\langle X, e_k\rangle+H^\alpha+\sum_{\beta,i,k} (-H^\beta+\lambda_\beta) h^\alpha_{ik} h^\beta_{ki}\right)\\
&=2|\nabla H|^2+2 H^2-2\sum_{\alpha,\beta,i,k}H^\alpha (H^\beta-\lambda_\beta) h^\alpha_{ik} h^\beta_{ik}+2\sum_{\alpha,k}H^\alpha H^\alpha_{,k}\langle X, e_k\rangle.\\
\end{align*}
\item
\begin{align*}
\Delta_f H^2 &= \Delta H^2 - \langle X, \nabla H^2\rangle\\
&= \Delta H^2-2\sum_{\alpha,k}H^\alpha H^\alpha_{,k}\langle X, e_k\rangle\\
&=2|\nabla H|^2+2H^2-2\sum_{\alpha,\beta,i,k}H^\alpha (H^\beta-\lambda_\beta) h^\alpha_{ik} h^\beta_{ik}\\
&=2|\nabla H|^2+2H^2-2\sum_{\alpha,\beta}H^\alpha (H^\beta-\lambda_\beta)\langle A^\alpha, A^\beta\rangle.
\end{align*}
\item To keep the formulas short,  summation convention is used in the proof. By Proposition 2.1 in \cite{xi1}
\begin{align*}
(\Delta A)_{ij} = &{\bf H}_{,ij} + \langle A_{ik}, {\bf H}\rangle A_{jk} -\langle A_{ij}, A_{kl}\rangle A_{kl} + 2\langle A_{il},A_{jk}\rangle A_{kl}\\
&-\langle 	A_{jk}, A_{kl}\rangle A_{il}- \langle A_{ik}, A_{kl}\rangle A_{jl}. 
\end{align*}
Denote
$$A_{ij} = (\overline \nabla _{e_i}e_j)^\perp = h^\alpha_{ij}e_\alpha,$$
and let  $S_{\alpha\beta} =h^\alpha_{ij}h^\beta_{ij}.$ Then $|A|^2 = 
\sum_\alpha S_{\alpha\alpha}.$

We  have
\begin{align*}
\langle\Delta A, A \rangle &= \langle {\bf H}_{,ij}, A_{ij}\rangle + \langle A_{ik}, {\bf H}\rangle \langle A_{jk},A_{ij}\rangle -\langle A_{ij}, A_{kl}\rangle \langle A_{kl}, A_{ij}\rangle \\
&+ 2\langle A_{il},A_{jk}\rangle \langle A_{kl}, A_{ij}\rangle-\langle 	A_{jk}, A_{kl}\rangle\langle A_{il}, A_{ij}\rangle- \langle A_{ik}, A_{kl}\rangle\langle A_{jl},A_{ij}\rangle\\
&= \langle {\bf H}_{,ij}, A_{ij}\rangle + \langle A_{ik}, {\bf H}\rangle \langle A_{jk},A_{ij}\rangle -\langle A_{ij}, A_{kl}\rangle \langle A_{kl}, A_{ij}\rangle \\
&+ 2\langle A_{il},A_{jk}\rangle \langle A_{kl}, A_{ij}\rangle-2\langle 	A_{jk}, A_{kl}\rangle\langle A_{il}, A_{ij}\rangle.
\end{align*}
Noting
\begin{align*}
&\langle A_{ij}, A_{kl}\rangle \langle A_{kl}, A_{ij}\rangle = h^\alpha_{kl}h^\alpha_{ij}h^\beta_{ij}h^\beta_{kl} =  \sum_{\alpha, \beta}S^2_{\alpha\beta},\\
&2\langle A_{il},A_{jk}\rangle \langle A_{kl}, A_{ij}\rangle
- 2\langle 	A_{jk}, A_{kl}\rangle\langle A_{il}, A_{ij}\rangle\\
&= 2 \sum_{\alpha\ne\beta}( A^{e_\beta}A^{e_\alpha}, A^{e_\alpha}A^{e_\beta} 
- 2 A^{e_\beta}A^{e_\alpha}, A^{e_\beta}A^{e_\alpha} ) =\sum_{\alpha\ne\beta}|[A^{e_\alpha}, A^{e_\beta}]|^2.
\end{align*}
Thus,
\begin{align*}
\langle\Delta A, A \rangle &= 	\langle {\bf H}_{,ij}, A_{ij}\rangle + \langle A_{ik}, {\bf H}\rangle\langle A_{il}, A_{kl}\rangle-\sum_{\alpha\ne\beta}|[A^{e_\alpha}, A^{e_\beta}]|^2-\sum_{\alpha,\beta}S^2_{\alpha\beta}.
\end{align*}
Therefore,
\begin{align}\label{22}
\Delta|A|^2 &=  |\nabla A|^2 +2\langle\Delta A, A \rangle \nonumber\\
&=2 |\nabla A|^2 + 	2\langle {\bf H}_{,ij}, A_{ij}\rangle + 2\langle A_{ik}, {\bf H}\rangle\langle A_{il}, A_{kl}\rangle-2\sum_{\alpha\ne\beta}|[A^{e_\alpha}, A^{e_\beta}]|^2-2\sum_{\alpha,\beta}S^2_{\alpha\beta}.
\end{align}
\item The summation convention is still used in the proof. From (\ref{19})
\begin{equation*}
H_{,j} = \langle 
X, e_k\rangle A_{j k}.
\end{equation*}
\begin{align}\label{23}
{\bf H}_{,ij} &= A_{ij} +\langle X, A_{ik}\rangle A_{jk}+ \langle X, e_k\rangle A_{jki}\nonumber\\ 
&= A_{ij} +\langle {\bf H_f}-{\bf H}, A_{ik}\rangle A_{jk}+ \langle X, e_k\rangle A_{ijk}
 \end{align}
 
 (\ref{22}) and (\ref{23}) yield
 \begin{align*}
\Delta_f|A|^2 &= 2 |\nabla A|^2 + 2|A|^2+	 2\langle {\bf H_f}, A_{ik}  \rangle\langle A_{il}, A_{kl}\rangle\nonumber\\
&-2\sum_{\alpha\ne\beta}|[A^{\alpha}, A^{\beta}]|^2-2\sum_{\alpha,\beta}S^2_{\alpha\beta}.
\end{align*}

\end{enumerate}
\end{proof}
\begin{theorem}
If  $|A|$ is bounded and satisfies
\begin{equation}\label{24}
|A|^2{|{\bf H}-{\bf H_f}|}\le{H},
\end{equation} 
then $\Sigma$ is one of the following:
\begin{enumerate}
\item  an $n$-plane $P$ with $d(O, P)=\lambda$ (see Example \ref{pla});
\item a round sphere $S^n(\sqrt r)\subset P,$ where $P$ is an $(n+1)$-plane, $d(O,P)=h<\lambda,$  $\displaystyle r=\frac{\pm \mu+\sqrt{\mu^2+4n}}2$ and $\mu =\sqrt{\lambda^2-h^2}$  (see Example \ref{sph} and \ref{sph1});
\item  a cylinder $S^k(\sqrt r) \times P, 1 \le k \le n,$ where $P$ is an $(n-k)$- linear subspace and $\displaystyle r=\frac {\pm \lambda+\sqrt{\lambda^2+4k}}2$ (see Example \ref{cyl}).
\end{enumerate}
\end{theorem} 
\begin{proof}
Since $|A|$ is bounded,  $h_{ij}^\alpha, |H|$ are bounded, too.
It follows that $H^2$ satisfies the condition in Theorem \ref{basis}, therefore
$$\int_\Sigma \Delta_fH^2 dV=0.$$
We have
\begin{align}\label{25}
0&=\int_\Sigma \Delta_fH^2 dV\nonumber\\
 &=\int_\Sigma (2|\nabla H|^2+2H^2-2\sum_{\alpha,\beta}H^\alpha (H^\beta-\lambda_\beta)\langle A^\alpha, A^\beta\rangle)dV\nonumber\\
&\ge\int_\Sigma (2|\nabla H|^2+2H^2-2\sum_{\alpha,\beta}H^\alpha (H^\beta-\lambda_\beta)|A^\alpha||A^\beta|)dV\\
&=\int_\Sigma (2|\nabla H|^2+2 H^2-2(\sum_{\alpha}H^\alpha |A^\alpha|)(\sum_{\beta}(H^\beta-\lambda_\beta) |A^\beta|))dV\nonumber\\
&\ge \int_\Sigma (2|\nabla H|^2+2 H^2-  2H|A|^2|{\bf H}-{\bf H_f}|)dV.\nonumber
\end{align}
Condition (\ref{24}) of the theorem together (\ref{25}) implies that 
\begin{equation}\label{26}
H^2-H|A|^2|{\bf H}-{\bf H_f}|=0,
\end{equation} 
and $\nabla H=0,$ i.e. $ H=\text {const.}.$  
We have two cases. 
\begin{enumerate}
\item  ${H=0}.$ We have  $|{\bf H}-{\bf H_f}|=|{\bf H_f}|=\lambda\ne 0.$  By (\ref{24}),  $|A|=0,$ i.e.  $\Sigma$ is totally geodesic and therefore
 an $n$-plane and $d(O,\Sigma)=\lambda$ (see Example \ref{pla}).
\item $H\ne 0.$ By (\ref{26}),  $|{\bf H}-{\bf H_f}|\ne0$ and $|A|\ne 0.$  
Because the equalities hold in (\ref{25}), $(H^\alpha), (H^\alpha-\lambda_\alpha)$ and $(A^\alpha)$ are proportional. It follows that
\begin{equation}\label{57}
H^\alpha=0,\ A^\alpha=0, \ \alpha\ne n+p,
\end{equation} 
\begin{equation}\label{58}
|A|^2 = \sum_{i,j}h^{n+p}_{ij}h_{i j}^{ n+p}=|A^{n+p}|^2\ne 0,
\end{equation}
\begin{equation}\label{59}
H=H^{n+p}\ne 0.
\end{equation}

Thus, ${\bf H}=H^{n+p}e_{n+p}$ and ${\bf H}-{\bf H_f}=(H^{n+p}-\lambda)e_{n+p}$ are constants and therefore $|A|$ is constant.

Now (\ref{26})  becomes
\begin{equation}\label{66}
H^{n+p}-|A|^2| H^{n+p}-\lambda|=0.
\end{equation}
From (\ref{47})
\begin{align}\label{60}
H^{n+p}_{,ij} = h^{n+p}_{ij} - (H^{n+p}-\lambda)h^{n+p}_{ik}h^{n+p}_{jk}+ \langle X, e_k\rangle h^{n+p}_{ijk}.
 \end{align}
Multiplying both sides of (\ref{60}) by $h^{n+p}_{ij} $ and summing over $i, j$ we get
\begin{align*}
\sum_{k,i,j}(H^{n+p}-\lambda) h^{n+p}_{ik} h^{n+p}_{kj}h^{n+p}_{ij} &=\sum_{k,i,j} h^{n+p}_{ijk} h^{n+p}_{ij}\langle X, e_k\rangle+ \sum_{i,j}h^{n+p}_{ij}h^{n+p}_{ij} \nonumber\\
&=\sum_k \frac 12A^{n+p}_{,k}\langle X, e_k\rangle h^{n+p}_{ij}+ \sum_{i,j}h^{n+p}_{ij}h^{n+p}_{ij} 
=|A|^2.
\end{align*}
Now we have,
\begin{align*}
\langle\Delta A, A \rangle &=\langle\Delta A^{n+p}, A^{n+p} \rangle\nonumber\\
 &= 	\sum_{i,j} { H}^{n+p}_{,ij} h_{ij}^{n+p}+  \sum_{i,j}H^{n+p}h^{n+p}_{ik} h^{n+p}_{il} h^{n+p}_{kl}-|[A^{n+p}, A^{n+p}]|^2-S^2_{(n+p)( n+p)}\\
 &=0+|A|^4-0-|A|^4=0.
\end{align*}
Therefore,
\begin{align}\label{33}
0&=\Delta|A|^2 = 2 |\nabla A|^2 + \langle\Delta A, A \rangle\nonumber\\
&=2 |\nabla A|^2.
\end{align}
 Because of (\ref{57}), by Theorem 1 of Yau in \cite{ya}, $\Sigma$ lies some
$(n + 1)$-dimensional linear subspace $\mathbb R^{n+1}.$  From (\ref{33}), $\nabla A \equiv 0$ and by  Theorem 4 of Lawson in \cite{la}, $\Sigma$ (up to isometries) must be  $S^k(r) \times \mathbb R^{n-k}, k=0,1,\ldots, n$. Furthermore, the  $\lambda$-submanifolds  equation (\ref{9}) implies that the $k$-dimensional sphere $S^k(r)$ should have the radius $r=\frac {\pm \lambda+\sqrt{\lambda^2+4k}}2$ (see Example \ref{sph1} and \ref{cyl}).
\end{enumerate}
 \end{proof}
\begin{rem}
\begin{enumerate}
\item When $\lambda=0,$ Condition (\ref{24}) becomes  
$$ A^2H\le{H},$$
the theorem was proved in \cite {cali}, Theorem 1.1 (see also Proposition 3.1 in \cite{dixi1}). 
\item When $p=1, $ the theorem  generalizes Theorem 1.2 in \cite{wepe} for the case $\lambda$-hypersurfaces. (see also  Theorem 1.2 in \cite{anmi}). 
\end{enumerate}
\end{rem}

 We have known that (Lemma \ref{lem15})
\begin{equation}\label{37}
\Delta_f|A|^2 = 2|\nabla A|^2 +2|A|^2 +2\langle\langle{\bf H_f}, A_{ik}\rangle A_{jk}, A_{ij}\rangle - 2 \sum_{\alpha\ne\beta}|[A^{e_\alpha}, A^{e_\beta}]|2 - 2 \sum_{\alpha,\beta}S^2_{\alpha\beta}
\end{equation}
  In codimension 1 case,  $ \sum_{\alpha\ne\beta}|[A^{e_\alpha}, A^{e_\beta}]|^2=0,$  $\sum_{\alpha,\beta}S^2_{\alpha\beta}=|A|^4$ and (\ref{37}) becomes (see Lemma 2.1 in \cite{gu}).
$$  \Delta_f|A|^2 = 2|\nabla A|^2 +2|A|^2 -2\lambda\langle A^2, A\rangle - 2 |A|^4.$$

We have
\begin{align}\label{38}
\langle{\bf H_f}, A_{ik}\rangle\langle A_{jk}, A_{ij}\rangle&=\lambda_\alpha A^\alpha_{ik} A^\beta_{jk} A^\beta_{ij}\nonumber\\
&= \lambda_\alpha A^\alpha_{ik} [A^\beta]^2_{ik}\nonumber\\
&\le [\sum_{\alpha, i,k}\lambda^2_\alpha (A^\alpha_{ik})^2]^{1/2} |A^2|\\
&\le |\lambda| |A| |A^2|= \lambda|A|^3. \nonumber
\end{align}
In general, by Lemma 5.3.1 in \cite{si}
\begin{equation}\label{39}
 \sum_{\alpha\ne\beta}|[A^{\alpha}, A^{\beta}]|^2+\sum_{\alpha,\beta}S^2_{\alpha\beta}\le (2-\frac 1p)|A|^4.
\end{equation}
Combining (\ref{37}), (\ref{38}), (\ref{39}), we have
\begin{equation}\label{40}
\Delta_f|A|^2 \ge 2|\nabla A|^2 + 2|A|^2 -2\lambda |A|^3- 2\left(2-\frac 1p\right)|A|^4.
\end{equation} 

 This is Simons' type inequality for  $\lambda$-submanifolds. When the codimension $m = 1,$ the above Simons type inequality for  $\lambda$-hypersurfaces   is much more simpler
$$
\Delta_f|A|^2 \ge 2|\nabla A|^2 + 2|A|^2 \left(1 - \lambda|A|-|A|^2\right).
$$ 
If 
\begin{equation}
|A| \le\frac{-\lambda+\sqrt{\lambda^2 + 4}}2,
\end{equation}
then $1 - \lambda|A|-|A|^2\ge 0.$
Therefore, apply Theorem \ref{basis} to $|A|^2$ we obtain 
\begin{equation}\label{aaa}
|\nabla A|^2=0,
\end{equation}
and 
\begin{equation}\label{bbb}
|A|^2(1 - \lambda|A|-|A|^2)=0.
\end{equation}
Because of (\ref{aaa}), by Theorem 4 of Lawson \cite{la}, Guang (Theorem 1.3, {\cite{gu}) claims that $\Sigma$ must be a product of a sphere and a linear space, i.e  either a round sphere $S^n,$ or 
a cylinder $S^k\times\mathbb R^{n-k}$ for $1 \le k \le n - 1,$ or  a hyperplane in $R^{n+1}.$

But because of (\ref{bbb}), $|A|^2=0,$ i.e $\Sigma$ is a hyperplane,  or $1 - \lambda|A|-|A|^2=0,$ i,e. $
|A| =\frac{-\lambda+\sqrt{\lambda^2 + 4}}2.$ In the later case, $\Sigma$ must be a cylinder $S^1(r) \times \mathbb R^{n-1},$ where $r=\frac {\lambda+\sqrt{\lambda^2 + 4}}2$  (see Examples \ref{sph} and \ref{cyl}). 

Thus, we can restate Theorem 1.3 in \cite{gu}, in a more precise form, as follows
 \begin{theorem} If $\Sigma^n\subset\mathbb R^{n+1}$ is a smooth complete embedded $\lambda$-hypersurfaces with polynomial volume growth, which satisfies
\begin{equation}|A| \le\frac{-\lambda+\sqrt{\lambda^2 + 4}}2,
\end{equation}
then $\Sigma$ is one of the following:
\begin{enumerate}
\item a hyperplane,
\item a cylinder $S^1(r) \times \mathbb R^{n-1},$ where $r=\frac {\lambda+\sqrt{\lambda^2 + 4}}2.$
  \end{enumerate}
 \end{theorem}

When the codimension $p \ge 2,$ inequality (\ref{39}) is refined  as follows (Theorem 1 in \cite{lili})
\begin{equation}\label{42}
 \sum_{\alpha\ne\beta}|[A^{\alpha}, A^{\beta}]|^2+\sum_{\alpha,\beta}S^2_{\alpha\beta}\le \frac 32|A|^4.
\end{equation}

In this case, we get
\begin{equation}
\Delta_f|A|^2 \ge 2|\nabla A|^2 + 2|A|^2 -2\lambda |A|^3- 3|A|^4,
\end{equation} 
and obtain the following.

  \begin{theorem}  If 
\begin{equation}|A| \le\frac{-\lambda+\sqrt{\lambda^2 + 6}}3,
\end{equation}
then $\Sigma$ is  an $n$-plane.
 \end{theorem} 
  
  \begin{proof}
Apply Theorem \ref{basis} to $|A|^2$ we imply that
  \begin{equation}|A|=\frac{-\lambda+\sqrt{\lambda^2 + 6}}3,
\end{equation}
and all inequalities in (\ref{38}) and (\ref{42}) become equalities. It implies that 
all of the matrices $A^1, A^2 \ldots,  A^p$ are  zero, i.e. $\Sigma$ is an $n$-plane.
 \end{proof}


\end{document}